\documentclass[11pt,twoside]{amsart}
\usepackage{amsmath, amsthm, amscd, amsfonts, amssymb, graphicx, color, cite, eucal, wasysym}
\usepackage[bookmarksnumbered, colorlinks, plainpages]{hyperref}

\newtheorem{Theorem}{Theorem}[section]
\newtheorem{Corollary}[Theorem]{Corollary}
\newtheorem{Lemma}[Theorem]{Lemma}

\newtheorem{Definition}[Theorem]{Definition}

\numberwithin{equation}{section}
\def\pn{\par\noindent}

\begin{document}

\leftline{ \scriptsize \it Bulletin of the Iranian Mathematical
Society  Vol. {\bf\rm XX} No. X {\rm(}201X{\rm)}, pp XX-XX.}

\vspace{1.3 cm}

\title{On weakly $\mathfrak{F}_{s}$-quasinormal subgroups of finite groups}
\author{Y. Mao, X. Chen$^*$ and W. Guo}

\thanks{{\scriptsize
\hskip -0.4 true cm MSC(2010): Primary: 20D20; Secondary: 20D10.
\newline Keywords: $\mathfrak{F}$-hypercenter, weakly $\mathfrak{F}_{s}$-quasinormal subgroups, Sylow subgroups, $p$-nilpotence, supersolubility.\\
$*$Corresponding author
\newline\indent{\scriptsize $\copyright$ 2014 Iranian Mathematical
Society}}}

\maketitle

\begin{center}
Communicated by\;
\end{center}

\begin{abstract} Let $\mathfrak{F}$ be a formation and $G$ a finite group. A subgroup $H$ of $G$ is said to be weakly $\mathfrak{F}_{s}$-quasinormal in $G$ if $G$ has an $S$-quasinormal  subgroup $T$ such that $HT$ is $S$-quasinormal in $G$ and $(H\cap T)H_{G}/H_{G}\leq Z_{\mathfrak{F}}(G/H_{G})$, where $Z_{\mathfrak{F}}(G/H_{G})$ denotes the $\mathfrak{F}$-hypercenter of $G/H_{G}$. In this paper, we study the structure of finite groups by using the concept of weakly $\mathfrak{F}_{s}$-quasinormal subgroup.
\end{abstract}

\vskip 0.2 true cm


\pagestyle{myheadings}
\markboth{\rightline {\scriptsize Mao, Chen and Guo}}
         {\leftline{\scriptsize On weakly $\mathfrak{F}_{s}$-quasinormal subgroups}}

\bigskip
\bigskip

\section{\bf Introduction}
\vskip 0.4 true cm

Throughout this paper, all groups considered are finite. $G$ always denotes a group, $\pi$ denotes a set of primes and $p$ denotes a prime. Let $|G|_{p}$ denote the order of Sylow $p$-subgroups of $G$. For any subgroup $H$ of $G$, we use $H_{G}$ and $H^{G}$ to denote the largest normal subgroup of $G$ contained in $H$ and the smallest normal subgroup of $G$ containing $H$, respectively.

A class of groups $\mathfrak{F}$ is called a formation if it is closed under taking homomorphic images and subdirect products. A formation $\mathfrak{F}$ is called saturated if $G\in \mathfrak{F}$ whenever $G/\Phi(G)\in \mathfrak{F}$. Also, a formation $\mathfrak{F}$ is said to be $S$-closed if every subgroup of $G$ belongs to $\mathfrak{F}$ whenever $G\in \mathfrak{F}$. The $\mathfrak{F}$-residual of $G$, denoted by $G^{\mathfrak{F}}$, is the smallest normal subgroup of $G$ with quotient in $\mathfrak{F}$. We use $\mathfrak{U}$, $\mathfrak{U}_p$ and $\mathfrak{N}_p$ to denote the formations of all supersoluble groups, $p$-supersoluble groups and $p$-nilpotent groups, respectively.

For a class of groups $\mathfrak{F}$, a chief factor $L/K$ of $G$ is said to be $\mathfrak{F}$-central in $G$ if $L/K\rtimes G/C_{G}(L/K)\in \mathfrak{F}$. A normal subgroup $N$ of $G$ is called $\mathfrak{F}$-hypercentral in $G$ if either $N=1$ or every chief factor of $G$ below $N$ is $\mathfrak{F}$-central in $G$. Let $Z_\mathfrak{F}(G)$ denote the $\mathfrak{F}$-hypercentre of $G$, that is, the product of all $\mathfrak{F}$-hypercentral normal subgroups of $G$. All unexplained notation and terminology are standard, as in \cite {KT,G2,HB}.

Recall that a subgroup $H$ of $G$ is said to be quasinormal (resp. $S$-quasinormal) in $G$ if $H$ permutes with every subgroup (resp. Sylow subgroup) of $G$. Let $\mathfrak{F}$ be a formation. Recently, Huang \cite{JH} introduced the concept of $\mathfrak{F}_{s}$-quasinormal subgroup: a subgroup $H$ of $G$ is said to be $\mathfrak{F}_{s}$-quasinormal in $G$ if $G$ has a normal subgroup $T$ such that $HT$ is $S$-quasinormal in $G$ and $(H\cap T)H_{G}/H_{G}\leq Z_{\mathfrak{F}}(G/H_{G})$. Also, Miao and Li \cite{LM} introduced the concept of $\mathfrak{F}$-quasinormal subgroup: a subgroup $H$ of $G$ is said to be $\mathfrak{F}$-quasinormal in $G$ if $G$ has a quasinormal subgroup $T$ such that $HT$ is quasinormal in $G$ and $(H\cap T)H_{G}/H_{G}\leq Z_{\mathfrak{F}}(G/H_{G})$. By using these two concepts, the authors obtained some interesting results on the structure of finite groups. As a continuation of the above ideas, we introduce the following weaker concept.

\begin{Definition} Let $\mathfrak{F}$ be a formation. A subgroup $H$ of $G$ is said to be weakly $\mathfrak{F}_{s}$-quasinormal in $G$ if $G$ has an $S$-quasinormal  subgroup $T$ such that $HT$ is $S$-quasinormal in $G$ and $(H\cap T)H_{G}/H_{G}\leq Z_{\mathfrak{F}}(G/H_{G})$. \end{Definition}

Note that not only the concepts of $\mathfrak{F}_{s}$-quasinormal subgroup and $\mathfrak{F}$-quasinormal subgroup, but also many other subgroup embedding properties are generalized by our concept (see Section 4 below). In this present paper, we study the properties of weakly $\mathfrak{F}_{s}$-quasinormal subgroups, and derive some criteria for a finite group to be $p$-nilpotent or supersoluble in terms of weakly $\mathfrak{F}_{s}$-quasinormal subgroups.

\section{\bf Preliminaries}
\vskip 0.4 true cm

\begin{Lemma}\label{2.1}\textup{\cite[Lemma 2.1]{Gu1}} Let $\mathfrak{F}$ be a non-empty saturated formation, $H\leq G$ and $N\unlhd G$. Then:

$(1)$ $Z_{\mathfrak{F}}(G)N/N\leq Z_{\mathfrak{F}}(G/N)$.

$(2)$ If $\mathfrak{F}$ is S-closed, then $Z_{\mathfrak{F}}(G)\cap H\leq Z_{\mathfrak{F}}(H)$.
\end{Lemma}

\begin{Lemma}\label{2.2} Let $H,K\leq G$ and $N\unlhd G$.

$(1)$ If $H$ is $S$-quasinormal in $G$, then $H$ is subnormal in $G$.

$(2)$ If $H$ is $S$-quasinormal in $G$, then $HN/N$ is $S$-quasinormal in $G/N$.

$(3)$ If $N\leq H$, then $H/N$ is $S$-quasinormal in $G/N$ if and only if $H$ is $S$-quasinormal in $G$.

$(4)$ If $H$ is $S$-quasinormal in $G$, then $H\cap K$ is $S$-quasinormal in $K$.

$(5)$ If $H$ is $S$-quasinormal in $G$, then $H/H_{G}$ is nilpotent.

$(6)$ If $H$ is a $p$-group, then $H$ is $S$-quasinormal in $G$ if and only if $O^{p}(G)\leq N_{G}(H)$.

$(7)$ If $H$ and $K$ are $S$-quasinormal in $G$, then $H\cap K$ is $S$-quasinormal in $G$.

\end{Lemma}

\begin{proof}
See \cite[Lemma 1.2.7, Theorem 1.2.14, Lemma 1.2.16 and Theorem 1.2.19]{Bal}.
\end{proof}

\begin{Lemma}\label{2.3}  Let $H\leq K\leq G$ and $N\unlhd G$. Then:

$(1)$ If $H$ is weakly $\mathfrak{F}_{s}$-quasinormal in $G$ such that $(|H|,|N|)=1$, then $HN/N$ is  weakly $\mathfrak{F}_{s}$-quasinormal in $G/N$.

$(2)$ $H/N$ is weakly $\mathfrak{F}_{s}$-quasinormal in $G/N$ if and only if $H$ is weakly $\mathfrak{F}_{s}$-quasinormal in $G$.

$(3)$ If $\mathfrak{F}$ is $S$-closed and $H$ is weakly $\mathfrak{F}_{s}$-quasinormal in $G$, then $H$ is weakly $\mathfrak{F}_{s}$-quasinormal in $K$.

\end{Lemma}

\begin{proof}

(1) Since $H$ is weakly $\mathfrak{F}_{s}$-quasinormal in $G$, $G$ has an $S$-quasinormal subgroup $T$ such that $HT$ is $S$-quasinormal in $G$ and $(H\cap T)H_G/H_{G}\leq Z_{\mathfrak{F}}(G/H_{G})$. It is easy to see that $HN\cap TN=(H\cap T)N$ for $(|H|,|N|)=1$. By Lemma \ref{2.2}(2), $TN/N$ and $HTN/N$ are $S$-quasinormal in $G/N$. Since $(H\cap T)H_G/H_{G}\leq Z_{\mathfrak{F}}(G/H_{G})$, $(H\cap T)(HN)_G/(HN)_G\leq Z_{\mathfrak{F}}(G/(HN)_G)$ by Lemma \ref{2.1}(1). This implies that $(HN/N\cap TN/N)(HN/N)_{G/N}/\\(HN/N)_{G/N}\leq Z_{\mathfrak{F}}((G/N)/(HN/N)_{G/N})$. Hence $HN/N$ is weakly $\mathfrak{F}_{s}$-quasinormal  in $G/N$.

(2) First suppose that $H/N$ is weakly $\mathfrak{F}_{s}$-quasinormal in $G/N$. Then $G/N$ has an $S$-quasinormal subgroup $T/N$ such that $(H/N)(T/N)$ is $S$-quasinormal in $G/N$ and $((H/N)\cap (T/N))(H/N)_{G/N}/(H/N)_{G/N}\leq Z_{\mathfrak{F}}((G/N)/(H/N)_{G/N})$. It follows that $T$ and $HT$ are $S$-quasinormal in $G$ by Lemma \ref{2.2}(3) and $(H\cap T)H_G/H_{G}\leq Z_{\mathfrak{F}}(G/H_{G})$. Hence $H$ is weakly $\mathfrak{F}_{s}$-quasinormal in $G$.
Now assume that $H$ is weakly $\mathfrak{F}_{s}$-quasinormal in $G$. Then a similar argument as in (1) shows that $H/N$ is weakly $\mathfrak{F}_{s}$-quasinormal in $G/N$.

(3) As $H$ is weakly $\mathfrak{F}_{s}$-quasinormal in $G$, $G$ has an $S$-quasinormal subgroup $T$ such that $HT$ is $S$-quasinormal in $G$ and $(H\cap T)H_G/H_{G}\leq Z_{\mathfrak{F}}(G/H_{G})$. Then $T\cap K$ and $H(T\cap K)$ are $S$-quasinormal in $K$ by Lemma \ref{2.2}(4). By Lemma \ref{2.1}(2), $(H\cap T)H_G/H_{G}\leq Z_{\mathfrak{F}}(G/H_{G})\cap (K/H_G)\leq Z_{\mathfrak{F}}(K/H_{G})$, and so $(H\cap T)H_K/H_{K}\leq Z_{\mathfrak{F}}(K/H_{K})$ by Lemma 2.1(1). Therefore, $H$ is weakly $\mathfrak{F}_{s}$-quasinormal in $K$.\end{proof}

\begin{Lemma}\label{2.6} \textup{\cite [Main Theorem]{FG}} Suppose that $G$ has a Hall $\pi$-subgroup and $2\notin \pi$. Then all the Hall $\pi$-subgroups are conjugate in $G$.
\end{Lemma}

Recall that a group $G$ is called $\pi$-closed if $G$ has a normal Hall $\pi$-subgroup. Moreover, a group $G$ is said to be a $C_\pi$-group if $G$ has a Hall $\pi$-subgroup and any two Hall $\pi$-subgroups of $G$ are conjugate in $G$.\par

\begin{Lemma}\label{2.4} \textup{\cite[Corollary 3.7]{Guo}} Let $P$ be a $p$-subgroup of $G$. Suppose that $G$ is a $C_\pi$-group with $p\notin \pi$. If every maximal subgroup of $P$ has a $\pi$-closed supplement in $G$, then $G$ is $\pi$-closed.\end{Lemma}

The next lemma is clear.

\begin{Lemma}\label{2.5} Let $p$ be a prime divisor of $|G|$ with $(|G|,p-1)=1$.

$(1)$ If $G$ has cyclic Sylow $p$-subgroups, then $G\in \frak{N}_p$.\par
$(2)$ If $N$ is a normal subgroup of $G$ such that $|N|_p\leq p$ and $G/N\in \frak{N}_p$, then $G\in \frak{N}_p$.\par
\end{Lemma}

\begin{Lemma}\label{2.7} \textup{\cite [Lemma 2.16]{AN}} Let $\mathfrak{F}$ be a saturated formation containing $\frak{U}$. Suppose that $N\unlhd G$ such that $G/N\in \mathfrak{F}$. If $N$ is cyclic, then $G \in \mathfrak{F}$.
\end{Lemma}

\section{\bf Main Results}
\vskip 0.4 true cm

\begin{Lemma}\label{1M} Let $P$ be a Sylow $p$-subgroup of $G$, where $p$ is a prime divisor of $|G|$ with $(|G|,p-1)=1$. If every maximal subgroup of $P$ either is weakly ${(\mathfrak{U}_p)}_{s}$-quasinormal or has a $p$-nilpotent supplement in $G$, then $G\in \frak{N}_p$.
\end{Lemma}

\begin{proof} Suppose that the result is false and let $G$ be a counterexample of minimal order. Then:

(1) $O_{p'}(G)=1$.

If $O_{p'}(G)>1$, then by Lemma \ref{2.3}(1), $G/O_{p'}(G)$ satisfies the hypothesis of the lemma. The choice of $G$ implies that $G/O_{p'}(G)\in \frak{N}_p$, and so $G\in \frak{N}_p$, a contradiction.

(2) \textit{$G$ is soluble.}

Assume that $G$ is not soluble. Then $p=2$ by the Feit-Thompson theorem.
If $O_{2}(G)>1$, then $G/O_{2}(G)$ satisfies the hypothesis of the lemma by Lemma \ref{2.3}(2). The choice of $G$ implies that $G/O_{2}(G)\in \frak{N}_2$. Thus $G$ is soluble. This contradiction shows that $O_{2}(G)=1$. If every maximal subgroup of $P$ has a $2$-nilpotent supplement in $G$, then $G$ has a Hall $2'$-subgroup. By Lemma \ref{2.6}, $G$ is a $C_{2'}$-group, and so $G\in \frak{N}_2$ by Lemma \ref{2.4}, which is impossible. Therefore, $P$ has a maximal subgroup $P_1$ that is weakly ${(\mathfrak{U}_2)}_{s}$-quasinormal in $G$.
Then $G$ has an $S$-quasinormal subgroup $T$ such that $P_{1}T$ is $S$-quasinormal in $G$ and $(P_{1}\cap T)(P_{1})_{G}/(P_{1})_{G}\leq Z_{\mathfrak{U}_2}(G/(P_1)_{G})$. Clearly, $(P_{1})_{G}\leq O_{2}(G)=1$. Then we have that $P_{1}\cap T\leq Z_{\mathfrak{U}_2}(G)$. Since $O_{2}(G)=O_{2'}(G)=1$ by (1), $Z_{\mathfrak{U}_2}(G)=1$, and so $P_{1}\cap T=1$. This implies that $|T|_{2}\leq 2$. Then by Lemma \ref{2.5}(1), $T\in \frak{N}_2$, and consequently $T\leq O_{2',2}(G)=1$ by Lemma \ref{2.2}(1). Thus $P_{1}$ is $S$-quasinormal in $G$. By Lemma \ref{2.2}(1) again, $P_1\leq O_2(G)=1$, and so $|G|_2\leq 2$. It follows that $G$ is soluble, a contradiction.

(3) \textit{$G$ has a unique minimal normal subgroup $N$, $G/N\in \frak{N}_p$ and $G=N\rtimes M$, where $M$ is a maximal subgroup of $G$. Moreover, $N=O_{p}(G)$ and $|N|>p$.}

Let $N$ be a minimal normal subgroup of $G$. Then by (1) and (2), $N$ is an elementary abelian $p$-group. By Lemma \ref{2.3}(2), the hypothesis of the lemma still holds for $G/N$. By the choice of $G$, $G/N\in \frak{N}_p$. Evidently, $N$ is the unique minimal normal subgroup of $G$ and $\Phi(G)=1$. Thus there exists a maximal subgroup $M$ of $G$ such that $G=N\rtimes M$. Since $C_G(N)\cap M=1$, $N=C_{G}(N)$, and thereby $N=O_p(G)$. If $|N|=p$, then by Lemma \ref{2.5}(2), $G\in \frak{N}_p$, a contradiction. Hence $|N|>p$.

(4) \textit{Final contradiction.}

Let $P_1$ be any maximal subgroup of $P$ such that $N\nleq P_1$.
Then $P=P_{1}N$, $(P_1)_G=1$ and $P_1>1$ by (3). Suppose that $P_1$ is weakly ${(\mathfrak{U}_p)}_{s}$-quasinormal in $G$. Then $G$ has an $S$-quasinormal subgroup $T$ such that $P_{1}T$ is $S$-quasinormal in $G$ and $P_{1}\cap T\leq Z_{\mathfrak{U}_p}(G)$.
If $Z_{\mathfrak{U}_p}(G)>1$, then $N\leq Z_{\mathfrak{U}_p}(G)$ by (3), and so $|N|=p$, which is impossible. Thus $Z_{\mathfrak{U}_p}(G)=1$. Then $P_{1}\cap T=1$, and we can conclude that $|T|_{p}\leq p$. If $T=1$, then $P_{1}$ is $S$-quasinormal in $G$. By (3) and Lemma \ref{2.2}(6), $N\leq (P_{1})^{G}=(P_{1})^P=P_{1}$. This contradiction shows that $T>1$. By Lemma \ref{2.5}(1), $T\in \frak{N}_p$. Let $T_{p'}$ be the normal $p$-complement of $T$. Then $T_{p'}$ is subnormal in $G$ by Lemma \ref{2.2}(1), and so $T_{p'}\leq O_{p'}(G)=1$ by (1). This implies that $T$ is a group of order $p$. Then $P_{1}T$ is a Sylow $p$-subgroup of $G$. By (3) and Lemma \ref{2.2}(1), $P=P_{1}T=O_{p}(G)=N$. Consequently, $N\leq T^{G}=T^{P}=T$ by (3) and Lemma \ref{2.2}(6), and so $|N|=p$, which contradicts (3). Therefore, $P_1$ has a $p$-nilpotent supplement in $G$. Since $G=N\rtimes M$ and $M\in \frak{N}_p$ by (3), every maximal subgroup of $P$ has a $p$-nilpotent supplement in $G$. Note that $G$ is a $C_{p'}$-group because $G$ is $p$-soluble. Then by Lemma \ref{2.4}, $G\in \frak{N}_p$. The final contradiction ends the proof.
\end{proof}

\begin{Theorem}\label{C1} Let $p$ be a prime divisor of $|G|$ with $(|G|, p-1)=1$ and $E$ a normal subgroup of $G$ such that $G/E\in \mathfrak{N}_p$. If $E$ has a Sylow $p$-subgroup $P$ such that every maximal subgroup of $P$ either is weakly $(\mathfrak{U}_p)_{s}$-quasinormal or has a $p$-nilpotent supplement in $G$, then $G\in \mathfrak{N}_p$.
\end{Theorem}

\begin{proof} By Lemmas \ref{2.3}(3) and \ref{1M}, $E\in \mathfrak{N}_p$. Let $E_{p'}$ be the normal $p$-complement of $E$. Then $E_{p'}\unlhd G$. Suppose that $E_{p'}>1$. Then by Lemma \ref{2.3}(1), we see that $G/E_{p'}$ satisfies the hypothesis of the theorem. Hence $G/E_{p'}\in \mathfrak{N}_p$ by induction on $|G|$, and so $G\in \mathfrak{N}_p$. We may, therefore, assume that $E_{p'}=1$. Then $E=P$ is a $p$-group. Let $V/P$ be the normal $p$-complement of $G/P$. By Schur-Zassenhaus Theorem, there exists a Hall ${p'}$-subgroup $V_{p'}$ of $V$ such that $V=P\rtimes V_{p'}$. Since $V\in \mathfrak{N}_p$ by Lemmas \ref{2.3}(3) and \ref{1M}, $V=P\times V_{p'}$. This induces that $V_{p'}$ is the normal $p$-complement of $G$. Consequently, $G\in \mathfrak{N}_p$.
\end{proof}

\begin{Lemma}\label{1N}  Let $P$ be a Sylow $p$-subgroup of $G$, where $p$ is a prime divisor of $|G|$. If $N_{G}(P)\in \mathfrak{N}_p$ and every maximal subgroup of $P$ either is weakly $(\mathfrak{U}_p)_{s}$-quasinormal or has a $p$-nilpotent supplement in $G$, then $G\in \mathfrak{N}_p$. \end{Lemma}

\begin{proof} If $p=2$, then obviously, $G\in \mathfrak{N}_2$ by Lemma \ref{1M}. So we only need to prove the lemma in the case that $p>2$. Suppose that the result is false and let $G$ be a counterexample of minimal order. Then:

(1) \textit{$O_{p'}(G)=1.$}

Suppose that $O_{p'}(G)>1$. Since $N_{G/O_{p'}(G)}(PO_{p'}(G)/O_{p'}(G))=N_{G}(P)\\O_{p'}(G)/O_{p'}(G)\in \mathfrak{N}_p$, $G/O_{p'}(G)$ satisfies the hypothesis of the lemma by Lemma \ref{2.3}(1). The choice of $G$ implies that $G/O_{p'}(G)\in \mathfrak{N}_p$, and thereby $G\in \mathfrak{N}_p$, a contradiction.

(2) \textit{If $P\leq H<G$, then $H\in \mathfrak{N}_p$.}

By Lemma \ref{2.3}(3), $H$ satisfies the hypothesis of the lemma, and so $H\in \mathfrak{N}_p$ by the choice of $G$.

(3) \textit{$G$ is $p$-soluble.}

Since $G\notin \mathfrak{N}_p$, then there exists a non-trivial characteristic subgroup $L$ of $P$ such that $N_{G}(L)\notin \mathfrak{N}_p$ by \cite[Chap. 8, Theorem 3.1]{DG}. If $L\ntrianglelefteq G$, then $P\leq N_G(L)<G$, and so $N_G(L)\in \mathfrak{N}_p$ by (2), which is impossible. Thus $L\unlhd G$. This implies that $O_{p}(G)>1$. Since $N_{G/O_{p}(G)}(P/O_{p}(G))=N_{G}(P)/O_{p}(G)\in \mathfrak{N}_p$, $G/O_{p}(G)$ satisfies the hypothesis of the lemma by Lemma \ref{2.3}(2). The choice of $G$ induces that $G/O_{p}(G)\in \mathfrak{N}_p$, and thereby $G$ is $p$-soluble.

(4) \textit{$G$ has a unique minimal normal subgroup $N$, $G/N\in \mathfrak{N}_p$ and $G=N\rtimes M$, where $M$ is a maximal subgroup of $G$. Moreover, $N=O_{p}(G)$ and $|N|>p$.}

Let $N$ be a minimal normal subgroup of $G$. Then by (1) and (3), $N\leq O_p(G)$. Since $N_{G/N}(P/N)=N_{G}(P)/N\in \mathfrak{N}_p$, the hypothesis of the lemma holds for $G/N$ by Lemma \ref{2.3}(2), and so $G/N\in \mathfrak{N}_p$ by the choice of $G$. It is easy to see that $N=O_p(G)$ is the unique minimal normal subgroup of $G$ and $G$ has a maximal subgroup $M$ such that $G=N\rtimes M$. If $|N|=p$, then by Lemma \ref{2.7}, $G\in \mathfrak{U}_p$. As $O_{p'}(G)=1$, $P\unlhd G$ by \cite[Lemma 2.1.6]{Bal}, and thus $G=N_G(P)\in \mathfrak{N}_p$, a contradiction. Hence $|N|>p$.

(5) \textit{Final contradiction.}

Let $P_1$ be any maximal subgroup of $P$ such that $N\nleq P_1$.
Then by (4), we have that $P=P_{1}N$, $(P_1)_G=1$ and $P_1>1$. Assume that $P_1$ is weakly ${(\mathfrak{U}_p)}_{s}$-quasinormal in $G$. Then $G$ has an $S$-quasinormal subgroup $T$ such that $P_{1}T$ is $S$-quasinormal in $G$ and $P_{1}\cap T\leq Z_{\mathfrak{U}_p}(G)$.
It follows from (4) that $Z_{\mathfrak{U}_p}(G)=1$. Otherwise $|N|=p$, a contradiction. Then $P_{1}\cap T=1$, and so $|T|_{p}\leq p$. If $T=1$, then $P_{1}$ is $S$-quasinormal in $G$. By (4) and Lemma \ref{2.2}(6), $N\leq (P_{1})^{G}=(P_{1})^P=P_{1}$, which is impossible. Thus $T>1$.
If $T_G>1$, then $N\leq T$ by (4), and so $|N|=p$, a contradiction. Hence $T_G=1$. By Lemma \ref{2.2}(5), $T$ is nilpotent. Since $T$ is subnormal in $G$ by Lemma \ref{2.2}(1), $T$ is a group of order $p$, because $O_{p'}(G)=1$ by (1). Then $P_{1}T$ is a Sylow $p$-subgroup of $G$. By (4) and Lemma \ref{2.2}(1), $P=P_{1}T=O_{p}(G)=N$. Thus $N\leq T^{G}=T^{P}=T$ by (4) and Lemma \ref{2.2}(6), and so $|N|=p$, which contradicts (4). Therefore, $P_1$ has a $p$-nilpotent supplement in $G$. Since $G=N\rtimes M$ and $M\in \mathfrak{N}_p$ by (4), every maximal subgroup of $P$ has a $p$-nilpotent supplement in $G$. Then by (3) and Lemma \ref{2.4}, $G\in \mathfrak{N}_p$. This is the final contradiction.
\end{proof}

\begin{Theorem}\label{C4}
Let $p$ be a prime divisor of $|G|$ and $E$ a normal subgroup of $G$ such that $G/E\in \mathfrak{N}_p$. If $E$ has a Sylow $p$-subgroup $P$ such that $N_G(P)\in \mathfrak{N}_p$ and every maximal subgroup of $P$ either is weakly $(\mathfrak{U}_p)_{s}$-quasinormal or has a $p$-nilpotent supplement in $G$, then $G\in \mathfrak{N}_p$. \end{Theorem}

\begin{proof} By Lemmas \ref{2.3}(3) and \ref{1N}, $E\in \mathfrak{N}_p$. Let $E_{p'}$ be the normal $p$-complement of $E$. Clearly, $E_{p'}\unlhd G$. Suppose that $E_{p'}>1$. Then by Lemma \ref{2.3}(1), $G/E_{p'}$ satisfies the hypothesis of the theorem. By induction on $|G|$, we have that $G/E_{p'}\in \mathfrak{N}_p$, and so $G\in \mathfrak{N}_p$. Hence we may assume that $E_{p'}=1$. Then $E=P$ is a $p$-group. Therefore, $G=N_G(P)\in \mathfrak{N}_p$.
\end{proof}

\begin{Theorem}\label{1s}  Suppose that for every prime divisor $p$ of $|G|$ and every non-cyclic Sylow $p$-subgroup $P$ of $G$, every maximal subgroup of $P$ either is weakly $(\mathfrak{U}_p)_{s}$-quasinormal or has a $p$-supersoluble supplement in $G$. Then $G\in\mathfrak{U}$.\end{Theorem}

\begin{proof} Suppose that the theorem is false and let $G$ be a counterexample of minimal order. Then:

(1) \textit{$G$ is a Sylow tower group of supersoluble type.}

Let $q$ be the smallest prime dividing $|G|$ and $Q$ a Sylow $q$-subgroup of $G$. If $Q$ is cyclic, then $G\in \mathfrak{N}_q$ by Lemma \ref{2.5}(1). Now suppose that $Q$ is non-cyclic. Since $G$ satisfies the hypothesis of Lemma \ref{1M}, $G\in \mathfrak{N}_q$ too. Then by Lemma \ref{2.3}(1), we can deduce that $G$ is a Sylow tower group of supersoluble type by analogy.

(2) \textit{Let $r$ be the largest prime dividing $|G|$ and $R$ a Sylow $r$-subgroup of $G$. Then $R$ is the unique minimal normal subgroup of $G$, $G/R\in\mathfrak{U}$ and $G=R\rtimes M$, where $M$ is a maximal subgroup of $G$. Moreover, $R=F(G)$ and $R$ is non-cyclic.}

By (1), $G$ is soluble and $R\unlhd G$. Let $N$ be any minimal normal subgroup of $G$. Then $N$ is elementary abelian. By Lemmas \ref{2.3}(1) and \ref{2.3}(2), the hypothesis of the theorem holds for $G/N$, and so the choice of $G$ implies that $G/N\in\mathfrak{U}$. Clearly, $N$ is the unique minimal normal subgroup of $G$ and $\Phi(G)=1$. It follows that $N\leq R$ and $G$ has a maximal subgroup $M$ such that $G=N\rtimes M$. Since $C_G(N)\cap M=1$, $N=C_{G}(N)$, and thereby $N=F(G)$. This induces that $R=N$. If $R$ is cyclic, then by Lemma \ref{2.7}, $G\in\mathfrak{U}$, which is impossible. Thus $R$ is non-cyclic.

(3) \textit{Final contradiction.}

Let $R_1$ be any maximal subgroup of $R$.
Then by (2), $(R_1)_G=1$ and $R_1>1$. Suppose that $R_1$ is weakly ${(\mathfrak{U}_r)}_{s}$-quasinormal in $G$. Then we can derive a contradiction as in step (5) of the proof of Lemma \ref{1N}. Hence $R_1$ has a $r$-supersoluble supplement in $G$, say $K$. Since $R\cap K\unlhd G$, by (2), either $R\cap K=1$ or $R\leq K$. In the former case, $R_1\cap K=1$, and so $R=R_1$, a contradiction. In the latter case, $G=K\in \mathfrak{U}_r$. Then $|R|=r$, a contradiction too. The proof is thus completed.
\end{proof}

\begin{Lemma}\label{2M} Let $P$ be a Sylow $p$-subgroup of $G$, where $p$ is a prime divisor of $|G|$ with $(|G|,p-1)=1$. If every cyclic subgroup of $P$ of order $p$ or $4$ $($when $P$ is a non-abelian $2$-group$)$ either is weakly ${(\mathfrak{U}_p)}_{s}$-quasinormal or has a $p$-nilpotent supplement in $G$, then $G\in \mathfrak{N}_p$.
\end{Lemma}

\begin{proof} Suppose that the result is false and let $G$ be a counterexample of minimal order.
Let $M$ be any maximal subgroup of $G$. By Lemma \ref{2.3}(3), it is easy to see that the hypothesis of the lemma still holds on $M$. Hence $M\in \mathfrak{N}_p$ by the choice of $G$, and so $G$ is a minimal non-$p$-nilpotent group. In view of \cite[Chap. IV, Satz 5.4]{HB} and \cite[Chap. VII, Theorem 6.18]{KT}, $G$ is a minimal
non-nilpotent group; $G=P\rtimes Q$, where $Q$ is a Sylow $q$-subgroup of $G$ with $q\neq p$;
$P/\Phi(P)$ is a chief factor of $G$; the exponent of $P$ is
$p$ or 4 (when $P$ is a
non-abelian $2$-group). If $P/\Phi(P)\leq Z_{\mathfrak{U}_p}(G/\Phi(P))$, then $G/\Phi(P)\in \mathfrak{U}_p$, and thereby $G\in \mathfrak{U}_p$. Since $(|G|,p-1)=1$, $G\in \mathfrak{N}_p$, which is impossible. Thus $P/\Phi(P)\nleq Z_{\mathfrak{U}_p}(G/\Phi(P))$, and so $|P/\Phi(P)|>p$.

Let $x\in P \backslash\Phi(P)$, $H=\langle x\rangle$ and $V=H\Phi(P)$. Then $|H|=p$ or $4$ (when $P$ is a
non-abelian $2$-group) and $H<P$. Since $P/\Phi(P)$ is a chief factor of $G$, $H_G\leq \Phi(P)$.
First suppose that $H$ is weakly ${(\mathfrak{U}_p)}_{s}$-quasinormal in $G$. Then $G$ has an $S$-quasinormal subgroup $T$ such that $HT$ is $S$-quasinormal in $G$ and $(H\cap T)H_G/H_G\leq Z_{\mathfrak{U}_p}(G/H_G)$. By Lemma \ref{2.2}(7), we may assume that $T\leq P$. Also, by Lemma \ref{2.1}(1), $(H\cap T)\Phi(P)/\Phi(P)\leq P/\Phi(P)\cap Z_{\mathfrak{U}_p}(G/\Phi(P))=1$, and so $T<P$. It follows from Lemma \ref{2.2}(6) that $T^G=T^P<P$. Since $P/\Phi(P)$ is a chief factor of $G$, $T\leq T^G\leq \Phi(P)$. Thus $V=HT\Phi(P)$ is $S$-quasinormal in $G$. By Lemma \ref{2.2}(6) again, we have that $P=V^G=V^P=V$. Hence $P=H$, a contradiction.
Now suppose that $H$ has a $p$-nilpotent supplement $K$ in $G$. Then $(P\cap K)\Phi(P)\unlhd G$. Since $P/\Phi(P)$ is a chief factor of $G$, $(P\cap K)\Phi(P)=P$ or $\Phi(P)$. If $P\leq K$, then $K=G$, and so $G\in \mathfrak{N}_p$, which is impossible. Thus $P\cap K\leq \Phi(P)$. This implies that $P=H(P\cap K)=H$, which is also impossible. The proof is thus finished.\end{proof}

\begin{Theorem}\label{C6} Let $p$ be a prime divisor of $|G|$ with $(|G|, p-1)=1$ and $E$ a normal subgroup of $G$ such that $G/E\in \mathfrak{N}_p$. If $E$ has a Sylow $p$-subgroup $P$ such that every cyclic subgroup of $P$ of order $p$ or $4$ $($when $P$ is a non-abelian $2$-group$)$ either is weakly $(\mathfrak{U}_p)_{s}$-quasinormal or has a $p$-nilpotent supplement in $G$, then $G\in \mathfrak{N}_p$.
\end{Theorem}

\begin{proof} Proof similarly as in Theorem \ref{C1} by using Lemma \ref{2M} instead of Lemma \ref{1M}.\end{proof}

\begin{Theorem}\label{2s} Let $E$ be a normal subgroup of $G$ such that $G/E\in \mathfrak{U}$. Suppose that for every prime $p$ dividing $|E|$ and every non-cyclic Sylow $p$-subgroup $P$ of $E$, every cyclic subgroup of $P$ of order $p$ or $4$ $($when $P$ is a non-abelian $2$-group$)$ either is weakly $(\mathfrak{U}_p)_{s}$-quasinormal or has a $p$-supersoluble supplement in $G$, then $G\in \mathfrak{U}$.\end{Theorem}

\begin{proof} Suppose that the result is false and let $G$ be a counterexample of minimal order.
A similar discussion as in the proof of Lemma \ref{2M} shows that $G$ is a minimal non-supersoluble group. In view of \cite[Theorem 12]{Ba1} and \cite[Chap. VII, Theorem 6.18]{KT}, $G$ is a soluble group that has a normal Sylow $p$-subgroup, say $G_p$; $G_p=G^\mathfrak{U}$;
$G_p/\Phi(G_p)$ is a chief factor of $G$; the exponent of $G_p$ is
$p$ or 4 (when $G_p$ is a
non-abelian $2$-group). Since $G/E\in \mathfrak{U}$, we have that $G_p\leq E$. If $|G_p/\Phi(G_p)|=p$, then by Lemma \ref{2.7}, $G/\Phi(G_p)\in \mathfrak{U}$, and so $G\in \mathfrak{U}$, which is impossible. Thus $|G_p/\Phi(G_p)|>p$. This implies that $G_p/\Phi(G_p)\nleq Z_{\mathfrak{U}_p}(G/\Phi(G_p))$.

Let $x\in G_p \backslash\Phi(G_p)$ and $H=\langle x\rangle$. Then $|H|=p$ or $4$ (when $G_p$ is a
non-abelian $2$-group) and $H<G_p$.
Suppose that $H$ is weakly ${(\mathfrak{U}_p)}_{s}$-quasinormal in $G$. Then we can get a contradiction similarly as in the proof of Lemma \ref{2M}. Now consider that $H$ has a $p$-supersoluble supplement $K$ in $G$. Then $(G_p\cap K)\Phi(G_p)\unlhd G$. Since $G_p/\Phi(G_p)$ is a chief factor of $G$, $(G_p\cap K)\Phi(G_p)=G_p$ or $\Phi(G_p)$. If $G_p\leq K$, then $K=G$, and so $G\in \mathfrak{U}_p$. This induces that $G_p\leq Z_{\mathfrak{U}}(G)$, and therefore $G\in \mathfrak{U}$, a contradiction. Thus $G_p\cap K\leq \Phi(G_p)$. Then $G_p=H(G_p\cap K)=H$, a contradiction too. The theorem is proved. \end{proof}

\section{\bf Some Applications of the theorems}
\vskip 0.4 true cm

Let $\mathfrak{F}$ be a formation. In Section 1, we observe that all $\mathfrak{F}_s$-quasinormal and $\mathfrak{F}$-quasinormal subgroups of $G$ are weakly $\mathfrak{F}_s$-quasinormal in $G$. Besides, recall that a subgroup $H$ of $G$ is said to be $c$-normal \cite {WY} in $G$ if $G$ has a normal subgroup $T$ such that $G=HT$ and $H\cap T\leq H_G$.
A subgroup $H$ of $G$ is called $\mathfrak{F}_n$-supplemented \cite{Yang} in $G$ if $G$ has a normal subgroup $T$ such that $G=HT$ and $(H\cap T)H_{G}/H_{G}\leq Z_{\mathfrak{F}}(G/H_{G})$.
A subgroup $H$ of $G$ is said to be $\mathfrak{F}_{h}$-normal \cite{Guo1} in $G$ if $G$ has a normal subgroup $T$ such that $HT$ is a normal Hall subgroup of $G$ and $(H\cap T)H_{G}/H_{G}\leq Z_{\mathfrak{F}}(G/H_{G})$.
A subgroup $H$ of $G$ is called $\mathfrak{F}_n$-normal \cite{G4} in $G$ if $G$ has a normal subgroup $T$ such that $HT$ is normal in $G$ and $(H\cap T)H_{G}/H_{G}\leq Z_{\mathfrak{F}}(G/H_{G})$. It is easy to see that all above-mentioned
subgroups of $G$ are also weakly $\mathfrak{F}_s$-quasinormal in $G$.

Therefore, many results in former literatures can be viewed as special cases of our theorems in Section 3, and we list some of them below:

\begin{Corollary}\textup{\cite[Theorem 3.4]{XG}} Let $p$ be the smallest prime dividing $|G|$ and $P$ a Sylow $p$-subgroup of $G$. If every maximal subgroup of $P$ is $c$-normal in $G$, then $G\in \mathfrak{N}_p$.\end{Corollary}

\begin{Corollary}\textup{\cite[Theorem 5.1]{Guo1}}  Let $p$ be a prime divisor of $|G|$ with $(|G|, p-1)=1$ and $P$ a Sylow $p$-subgroup of $G$. Then $G\in \mathfrak{N}_p$ if and only if every maximal subgroup of $P$ is $\mathfrak{U}_h$-normal in $G$.\end{Corollary}

\begin{Corollary}\textup{\cite[Theorem 4.2]{G4}}  Let $p$ be a prime divisor of $|G|$ with $(|G|, p-1)=1$ and $P$ a Sylow $p$-subgroup of $G$. Then $G\in \mathfrak{N}_p$ if and only if every maximal subgroup of $P$ not having a $p$-nilpotent supplement in $G$ is $\mathfrak{U}_{n}$-normal in $G$.\end{Corollary}

\begin{Corollary}\textup{\cite[Theorem 3.2]{Hua1}} Let $p$ be a prime divisor of $|G|$ with $(|G|, p-1)=1$. Assume that $G$ has a normal subgroup $N$ such that $G/N\in \mathfrak{N}_p$ and for every maximal subgroup $M$ of each Sylow
$p$-subgroup of $N$ which is not ${(\mathfrak{N}_p)}_s$-quasinormal in $G$, $M$ has a $p$-nilpotent supplement in $G$. Then $G\in \mathfrak{N}_p$.\end{Corollary}

\begin{Corollary}\textup{\cite[Lemma 2.7]{LM}}  Let $p$ be the smallest prime divisor of $|G|$ and $P$ a Sylow $p$-subgroup of $G$. Then $G\in \mathfrak{N}_p$ if and only if every maximal subgroup of $P$ having no $p$-nilpotent supplement in $G$ is $\mathfrak{N}_p$-quasinormal in $G$.\end{Corollary}

\begin{Corollary}\textup{\cite[Theorem 3.1]{XG}} Let $p$ be an odd prime dividing $|G|$ and $P$ a Sylow $p$-subgroup of $G$. If $N_{G}(P)\in \mathfrak{N}_p$ and every maximal subgroup of $P$ is $c$-normal in $G$, then $G\in \mathfrak{N}_p$.\end{Corollary}

\begin{Corollary}\textup{\cite[Theorem 5.2]{Guo1}} Let $p$ be a prime divisor of $|G|$ and $P$ a Sylow $p$-subgroup of $G$. Then $G\in \mathfrak{N}_p$ if and only if $N_{G}(P)\in \mathfrak{N}_p$ and every maximal subgroup of $P$ is $\mathfrak{U}_h$-normal in $G$.\end{Corollary}

\begin{Corollary}\textup{\cite[Theorem 4.3]{G4}}  Let $p$ be a prime divisor of $|G|$ and $P$ a Sylow $p$-subgroup of $G$. Then $G\in \mathfrak{N}_p$ if and only if $N_{G}(P)\in \mathfrak{N}_p$ and every maximal subgroup of $P$ not having a $p$-nilpotent supplement in $G$ is $\mathfrak{U}_{n}$-normal in $G$.\end{Corollary}

\begin{Corollary}\textup{\cite[Theorem 4.1]{WY}} Suppose that $P_1$ is $c$-normal in $G$ for every Sylow subgroup $P$ of $G$ and every maximal subgroup $P_1$ of $P$. Then $G\in \frak{U}$. \end{Corollary}

\begin{Corollary}\textup{\cite[Corollary 3.8]{Yang}} $G\in \mathfrak{U}$ if and only if every maximal subgroup of every non-cyclic Sylow subgroup of $G$ is $\mathfrak{U}_n$-supplemented in $G$.\end{Corollary}

\begin{Corollary}\textup{\cite[Lemma 3.8]{Ram}} Let $p$ be the smallest prime dividing $|G|$ and $P$ a Sylow
$p$-subgroup of $G$. If the subgroups of $P$ of order $p$ or order $4$ are $c$-normal in $G$, then $G\in \mathfrak{N}_p$.\end{Corollary}

\begin{Corollary}\textup{\cite[Theorem 3.3]{Hua1}} Let $p$ be a prime divisor of $|G|$ with $(|G|, p-1)=1$. Assume that $G$ has a normal subgroup $N$ such that $G/N\in \mathfrak{N}_p$ and for every cyclic subgroup $L$ of order $p$ or $4$ of $N$ which is not ${(\mathfrak{N}_p)}_s$-quasinormal in $G$, $L$ has a $p$-nilpotent supplement in $G$. Then $G\in \mathfrak{N}_p$.\end{Corollary}





\begin{Corollary}\textup{\cite[Theorem 4.2]{WY}} Suppose that $\langle x\rangle $ is $c$-normal in $G$ for every element $x$ of $G$ with prime order or order $4$. Then $G\in \frak{U}$.\end{Corollary}

\begin{Corollary}\textup{\cite[Corollary 3.6]{Guo1}} $G\in \mathfrak{U}$ if and only if every cyclic subgroup of $G$
of prime order or order $4$ is $\mathfrak{U}_h$-normal in $G$.\end{Corollary}

\vskip 0.4 true cm

\begin{center}{\textbf{Acknowledgments}}
\end{center}
Research is supported by a NNSF grant of China (grant \#11371335) and Research Fund for the Doctoral Program of Higher Education of China (Grant 20113402110036). The first author is supported by
the Youth Foundation of Shanxi Datong University (2012Q17). \\ \\
\vskip 0.4 true cm



\bigskip
\bigskip


{\footnotesize \pn{\bf Yuemei Mao}\; \\ {Department of Mathematics}, {University of Science and Technology of China, Hefei, 230026, P. R. China}\\ {School of Mathematical and Computer Sciences}, {University of Datong of Shanxi, Datong, 037009, P. R. China}\\
{\tt Email: maoym@mail.ustc.edu.cn}\\

{\footnotesize \pn{\bf Xiaoyu Chen}\; \\ {Department of Mathematics}, {University of Science and Technology of China, Hefei, 230026, P. R. China}\\
{\tt Email: jelly@mail.ustc.edu.cn}\\

{\footnotesize \pn{\bf Wenbin Guo}\; \\ {Department of Mathematics}, {University of Science and Technology of China, Hefei, 230026, P. R. China}\\
{\tt Email: wbguo@ustc.edu.cn}\\
\end{document}